\documentclass[a4paper,12pt]{amsart}

\usepackage[all]{xy}
\usepackage{graphics}
\usepackage{color}
\usepackage[T1]{fontenc}

\usepackage[upright]{fourier}
\usepackage{kerkis}
\usepackage{eucal}
\usepackage{parskip}
\newcommand{\B}[1]{{\mathbf #1}}
\newcommand{\C}[1]{{\mathcal #1}}
\newcommand{\F}[1]{{\mathfrak #1}}

\usepackage[colorlinks]{hyperref}
\usepackage{showlabels}

\newtheorem{theorem}[equation]{Theorem}
\newtheorem{corollary}[equation]{Corollary}
\newtheorem{lemma}[equation]{Lemma}
\newtheorem{proposition}[equation]{Proposition}
\theoremstyle{definition}
\newtheorem{example}[equation]{Example}
\theoremstyle{remark}

\newtheorem{remark}[equation]{Remark}

\numberwithin{equation}{section}
\numberwithin{figure}{section}
\numberwithin{table}{section}

\newcommand\OP{\operatorname}
\newcommand\Diff{\OP{Diff}}
\newcommand\Homeo{\OP{Homeo}}
\newcommand\Hom{\OP{Hom}}

\newcommand\PSl{\OP{PSl}}

\title[$\int_x^{hx}g^*\alpha - \alpha$]{
\boldmath $\displaystyle\int\limits_{x}^{hx}g^*\alpha - \alpha$}
\author{\'Swiatos\l aw R. Gal}
\address{\'S.G. --- Uniwersytet Wroc\l awski \& Uniwersit\"at Wien}
\email{sgal@math.uni.wroc.pl}
\author{Jarek K\k edra}
\address{J.K. --- University of Aberdeen \& Uniwersytet Szczeci\'nski}
\email{kedra@abdn.ac.uk}

\begin{document}

\begin{abstract}
Let $X$ be a path-connected topological space admitting a universal cover.
Let $\F a$ be a degree one cohomology class on~$X$. We define and study
a two-cocycle on a group acting on $X$ by homeomorphisms
preserving the class $\F a$.

We apply the cocycle to the study of the distortion in the group
of homeomorphisms preserving the class $\F a$.
In particular, we introduce a local rotation number
of a homeomorphism and prove that a homeomorphism with non-constant
local rotation number is undistorted. We also use the cocycle to
investigate group actions on~$X$.  For example, we show that if an
action preserves a Borel probability measure on $X$ then the cocycle
is cohomologically trivial.
\end{abstract}

\maketitle

\section{Introduction and the statement of the results}
\label{S:intro}

Let $X$ be a path-connected topological space admitting a universal
cover. Let $\Homeo(X,\F a)$ denote the group of homeomorphisms
of $X$ preserving a cohomology class $\F a\in H^1(X;\B A)$, where
$\B A$ is a trivial coefficient system. In the present paper,
we study distortion in $\Homeo(X,\F a)$ and properties of group
acions on $X$ by homeomorphisms preserving the class $\F a$.

The main tool is a two-cocycle on the group
$\Homeo(X,\F a)$ defined by the following formula
$$
\F G_{x,\alpha}(g,h):=\int_\gamma g^*\alpha - \alpha,
$$
where $x\in X$ is a reference point, $\gamma$ is a
path from $x$ to $hx$, and $\alpha$ is a singular cocycle
representing the class $\F a\in H^1(X,\B A)$. The
expression $\int_\gamma\sigma$ denotes the natural
pairing between a chain $\gamma$ and a cochain $\sigma$.
We shall frequently omit one or both subscripts when 
it does not lead to a confusion.

The results obtained in the paper can be divided into two parts. 
The first (Section \ref{S:distortion}) is about distortion. 
In the second part (Sections 
\ref{S:vanish}--\ref{S:nonvanish})
we study the cohomology class of $\F G$ and properties
of group actions on $X$ depending on either vanishing or
non-vanishing of the cohomology class of $\F G$. 

Throughout the paper we consider groups of homeomorphisms
equip\-ped either with compact-open or discrete topology.
The latter case is marked with the superscript $\delta$.
By $G_0$ we denote the connected component of the
identity of $G$.
Let us now discuss the main results.

\subsection{Distortion in groups}

Let $\Gamma$ be a finitely generated group.  Define the word norm
associated with fixed set of generators $S$ to be
$$
|g|:=\min\{k\in \B N\,|\,g=s_1\ldots s_k,\,s_i\in S\}.
$$
The {\bf translation length} 
of an element $g\in \Gamma$ is defined to be
$$
\tau(g):=\lim_{n\to \infty}\frac{|g^n|}{n}.
$$
(Note that by the subadditivity property of the length, the limit exists.) 
An element $g\in \Gamma$ is called {\bf undistorted} if
its translation length is positive and this property
does not depend on the choice of generators. 
If $G$ is a general (not necessarily finitely generated)
group then $g\in G$ is called {\bf undistorted} if it
is undistorted in every finitely generated subgroup
of $G$. Notice that distortion in a subgroup implies 
distortion in an ambient group.

It is well known that
certain lattices in semisimple Lie groups contain distorted
elements due to a result of Lubotzky, Mozes, and Raghunathan
\cite{MR1828742}. On the other hand, the distortion in
groups of diffeomorphisms of closed manifolds is rare
as shown, for example,  by Franks and Handel \cite{MR2219247}, 
Gambaudo and Ghys \cite{MR2104597}, or Polterovich \cite{MR2003i:53126}. 
This provides restrictions on possible actions of such lattices.

The papers cited above are concerned with the distortion either
in volume preserving or in Hamiltonian diffeomorphisms.
It follows from our results, however, that many elements are undistorted
in groups of homeomorphisms of manifolds of dimension at least two
and with nontrivial first real cohomology. Essentially, this is as much as
one gets for such manifolds. In contrast, Calegari and Freedman 
\cite[Theorem C]{MR2207794} proved that all homeomorphisms of the 
sphere $\B S^n$ are distorted in $\Homeo(\B S^n)$.

The results presented in this section are consequences of a more general
Theorem \ref{T:q-m}. The proof of the
following theorem is presented in Section \ref{SS:proof_distortion}.

\begin{theorem}\label{T:distortion}
Let $X$ be compact and let $g\in G\subseteq \Homeo(X,\F a)$ 
where $\F a\in H^1(X;\B R)$ is represented by a one-cocycle $\alpha$. 
Suppose that $g$ has two fixed points $x,y\in X$ and let $\gamma$ 
be a path from $x$ to $y$. If 
$$
\int_\gamma g^*\alpha-\alpha \neq 0
$$
then $g$ is undistorted in $G$.
\end{theorem}

Two fixed points $x,y$ of a map $g\colon X\to X$
are called {\bf Nielsen equivalent} if there exists a path
$\gamma$ from $x$ to $y$ such that $\gamma$ and $g\gamma$
are homotopic modulo the endpoints. The hypothesis
of the above theorem implies that the homeomorphism
$g$ has two fixed points which are Nielsen nonequivalent
in a stronger sense. Namely, the cycle $g\gamma - \gamma$
is homologically nontrivial.

\begin{example}\label{E:surface_nielsen}
Let $G\subset \Homeo(\Sigma)$ be a group of homeomorphisms
of a closed oriented surface $\Sigma$ acting trivially
on the first cohomology of~$\Sigma$. Suppose that  $g\in G$ has
two fixed points $x,y\in \Sigma$ such that
$g\gamma - \gamma$ is a homologically nontrivial loop,
where $\gamma$ is a path from $x$ to $y$. 
Then $g$ is undistorted in $G$. Indeed, there exists a
cohomology class $\F a\in H^1(\Sigma;\B Z)$ evaluating
nontrivially on $g\gamma - \gamma$.
\hfill $\diamondsuit$
\end{example}

The cocycle $\F G$ depends on a reference point and
it is unbounded in general. In Section \ref{SS:bounded} we define
a local rotation number of a homeomorphism with
respect to a point at which the cocycle $\F G$ is
bounded. 

\begin{theorem}\label{T:rotation}
Let $\F a\in H^1(X;\B Z)$ and let $g\in \Homeo(X,\F a)$
and assume that $X$ is compact.
Let $x$ and $y$ be points such that the cocycles
$\F G_x$ and $\F G_{y}$ are bounded on the cyclic subgroup
generated by $g$. If the local rotation numbers of $g$
at $x$ and $y$ are distinct then $g$ is undistorted
in $\Homeo(X,\F a)$.
\end{theorem}

\begin{example}\label{E:annulus}
Let $X=\B S^1\times [0,1]$ be the closed annulus and let
$h\colon X\to X$ be a homeomorphism preserving the orientation
and the components of the boundary. If the topological rotation numbers
of $h$ restricted to the boundary circles are distinct then
$h$ is undistorted in the group of orientation preserving
homeomorphisms of $X$. 
\hfill
$\diamondsuit$
\end{example}

We discuss more applications of Theorem \ref{T:rotation}
in Section~\ref{SS:top_rot}.

\subsection{Vanishing properties and dynamics}
Let us discuss the conditions under which the cohomology 
class of $\F G$ is trivial.
This has a dynamical flavour as nonvanishing of $[\F G]$ is an obstruction
to the existence of certain invariant objects. The proof of the
next result is presented in Section \ref{SS:proof_subset}.

\begin{theorem}\label{T:subset}
Let $i_L\colon L\to X$ be the inclusion of a path-connected subset
such that $i_L^*\F a=0\in H^1(L;\B A)$. If an action
$\psi\colon G\to\Homeo(X,\F a)$ preserves the subset\/ $L$
then $\psi^*[\F G]=0$.
\end{theorem}

One can think of this result as a form of ergodicity in
which invariant subsets have necessarily complicated
topology. In particular, if $\psi^*[\F G]$ is nontrivial then
the action preserves no path-connected and simply-connected
subsets.

\begin{theorem}\label{T:measure}
Suppose that $X$ is compact.
If the action $\psi\colon G\to \Homeo(X,\F a)$ preserves
a Borel probability measure on $X$ then the class
$\psi^*[\F G]\in H^2(BG^{\delta};\B R)$ is trivial.
\end{theorem}
The proof is presented in Section \ref{SS:proof_measure}.

A topological group $G$ is called {\bf amenable} if
a continuous affine action of $G$ on a non-empty
compact convex subset of a locally convex topological
vector space has a fixed point \cite[Theorem G.1.7]{MR2415834}.

The space of  Borel probability measures on a compact
space is compact as a subset of the dual of the Banach space
of continuous functions on $X$ equipped with weak-* topology.
Consequently, every action of an amenable group on a compact
space preserves a Borel probability measure. This proves the
following result.

\begin{corollary}\label{C:amenable}
Let $X$ be compact.  If\/ $\psi\colon G\to
\Homeo(X,\F a)$ is an action of a topological amenable group
then $\psi^*[\F G]=0$ in $H^2(BG^\delta;\B R)$.
\qed
\end{corollary}

\begin{example}
Let $X=U(1)$ and consider the natural action of $U(1)$ on itself.
Then $[\F G]$ is nontrivial in $H^2(BU(1)^{\delta};\B Z)$ according
to Theorem~\ref{T:def} and trivial in $H^2(BU(1)^{\delta};\B R)$
by the above corollary.
\hfill $\diamondsuit$
\end{example}

\subsection{Nonvanishing properties}
Let $\OP{ev}\colon G\to X$ denote the evaluation at the reference
point $x\in X$ associated with an action of a topological group
$G$ on $X$. Our main result about the nonvanishing
of the class $[\F G]$ is concerned with the homology 
of the evaluation map (see Section~\ref{SS:proof_ev} for the
proof).

\begin{theorem}\label{T:ev}
Let\/ $\psi\colon G\to \Homeo(X,{\F a})$ be an action of a
connected topological group. Assume that
the homomorphism
$
H^2(BG;\B A)\to H^2(BG^\delta;\B A)
$
induced by the identity on $G$ is injective.
Then the class $\psi^*[\F G]\in H^2(BG^\delta;\B A)$ is nonzero
if and only if\/
$\OP{ev}^*({\F a})\in H^1(G;\B A)$ is nontrivial.
\end{theorem}

There are cases in which the hypothesis of Theorem \ref{T:ev}
is satisfied. For example, let $X$ be a connected topological manifold with
trivial ends. It is a result of McDuff \cite{MR569248}
that the identity homomorphism induces an isomorphism
$H^*(BG;\B A)\to H^*(BG^{\delta};\B A)$,
where $G$ is a group of homeomorphisms of $X$
containing the component of the identity.

\begin{corollary}\label{C:homeo}
Let $X$ be a connected topological manifold with trivial ends.
Consider the natural action of\/ $G=\Homeo(X)_0$,
the connected component of the identity of the group
of homeomorphisms of\/ $X$.
Then the class $[\F G] \in H^2(BG^{\delta};\B A)$ is
nonzero if and only if\/ $\OP{ev}^*(\F a)\neq 0$.
\qed
\end{corollary}

Another instance where the hypothesis of Theorem \ref{T:ev}
is satisfied is when $G$ is a connected perfect group with countable
fundamental group.  The following result is proven in
Section \ref{SS:perfect}.

\begin{corollary}\label{C:perfect}
Let $\psi\colon G\to \Homeo(X,{\F a})$ be an action of
a connected perfect group with countable fundamental group.
Then $\psi^*[\F G]\neq 0$ if and only if  $\OP{ev}^*({\F a})\neq 0$.
\end{corollary}

Let $K$ be a connected compact Lie group.  
Then the evaluation map of
the action of $K$ on itself induces an isomorphism
$\OP{ev}^*\colon H^1(K;\B A)\to H^1(K;\B A)$
for any coefficients $\B A$.  If there exist a nontrivial 
homomorphism $\pi_1(K)\to \B A$
then $H^1(K;\B A)=\Hom(\pi_1(K);\B A)$ is nontrivial.
Let $G$ be a connected algebraic semisimple Lie group
and let $G=KAN$ be its Iwasawa decomposition.
Then $G$ acts on $K=G/AN$ and the action extends the natural
action of\/ $K$ on itself. Since $G$ is homotopy equivalent
to $K$ we obtain that the action induces an isomorphism               
$\OP{ev}^*\colon H^1(G;\B A)\to H^1(K;\B A)$.
Applying this observation and Corollary
\ref{C:perfect} gives the following application.

\begin{corollary}
Let $\psi\colon G\to \Homeo_0(K)$ be the action of a connected 
algebraic semisimple Lie group on its maximal compact subgroup 
$K\subset G$. Then the  class $\psi^*[\F G_{\alpha}]\in H^2(BG^{\delta};\B A)$ 
associated with the action is nontrivial for any nonzero element 
$[\alpha]\in H^1(K;\B A)$.
\qed
\end{corollary}

\begin{example}
Let $\Gamma \subset G$ be a torsion free uniform lattice in
a connected non-compact simple Lie group of Hermitian type.
Then $B \Gamma = \Gamma\backslash G/K$ is a closed K\"ahler
manifold and the homomorphism
$H^2(BG^{\delta};\B R)\to H^2(B\Gamma;\B R)$
is injective.
In this case the class of the cocycle $\F G$ is nontrivial
in $H^2(B\Gamma;\B R)$
with respect to the action of $\Gamma$ on $K$ (as in the above
corollary).  The evaluation homomorphism
is necessarily trivial since $\Gamma$ is discrete.
\hfill $\diamondsuit$
\end{example}

\begin{example}\label{E:maps}
Let $X=\OP{Map}_1(\B S^1,\B S^1)$ be the space of continuous
degree one self-maps of the circle. The group $G=\Homeo(\B S^1)_0$
acts on $X$ by the reparametrisations. Let $\F a\in H^1(X;\B Z)$
be the class equal to the pull back of a generator of
$H^1(\B S^1;\B Z)$ with respect to the evaluation at a point.
Then $\OP{ev}^*(\F a)\neq 0$ because the composition
$\OP{SO}(2)\subset G\stackrel{\OP{ev}}\to X\to S^1$ is
equal to the identity.
Since, moreover, the group $G$ is perfect, Corollary
\ref{C:perfect} applies and the corresponding cocycle
$\F G$ is cohomologically nontrivial on $G$.
\hfill $\diamondsuit$
\end{example}

Let us finish this section with an example of
a tautological construction providing actions with
nontrivial class of the cocycle $\F G$. 

\begin{example}
Consider a central extension of discrete groups
$$
0\to \B A\to \widehat G \to G\to 1
$$
and observe that $G$ acts on the space $X:=\B A\backslash E\widehat G$.
It follows from Theorem \ref{T:def} that the cohomology class
of the cocycle $\F G$ associated with this action is equal to
the class $\F E\in H^2(BG^{\delta};\B A)$ of the above central
extension. This shows that every class in the second cohomology
of a discrete group is equal to the class of the cocycle $\F G$
associated with an appropriate action.
\hfill $\diamondsuit$
\end{example}

\subsection*{Historical remarks}
The cocycle $\F G$ can be defined  for
an arbitrary, not necessarily closed, one-cochain $\alpha$
on a suitably defined subgroup of $\Homeo(X)$. It has been first
defined by Is\-ma\-gi\-lov, Losik, and Michor \cite{MR2270616} for a
primitive of a symplectic form and further studied by the
authors in \cite{GK2}.

The cocycle $\F K_{\alpha}$ (see Section \ref{SS:one-cocycle} for definition)
appears in Gambaudo and Ghys \cite{MR1452855} and in Arnold and Khesin
\cite[p.~247]{MR1612569} in the case of a symplectic ball.
It has been studied for a general symplectically aspherical
manifold in \cite{GK1}.

The local rotation number generalizes the rotation number 
of a homeomorphism of a circle.  There are related notions 
in the literature.  For example the rotation defined by 
Burger, Iozzi, and Wienhard in \cite[Definition 7.1]{MR2680425}
or the rotation vector of Franks \cite[Definition 2.1]{MR1325916}.

\section{Cocycles on $\Homeo(X,\F a)$}\label{S:cocycles}
In this section we discuss the basic properties of the cocycle $\F G$.
We also define an auxiliary one-cocycle $\F K_\alpha$ with values in
the space of functions on $X$ modulo constants.  The main result of
this section is Lemma \ref{L:continuous} stating that under suitable 
assumptions $\F K_\alpha$ takes values in the space of {\em continuous} 
functions on $X$ modulo constants.

Let $X$ be a path-connected, topological space admitting
a universal cover $\widetilde X\to X$.
Let ${\F a}\in H^1(X;\B A)$ be a cohomology class, where
$\B A$ is an Abelian group of trivial coefficients.
Let $G\subseteq \Homeo(X,{\F a})$ be a group of homeomorphisms
of $X$ preserving the class ${\F a}$.

\subsection{An explicit formula for the cocycle $\F G$}\label{SS:explicit}
Let $\alpha\in Z^1(X;\B A)$ be a singular one-cocycle representing
the class $\F a$. Let $x\in X$ be a reference point.
Define $\F G_{x,\alpha}\colon G\times G\to \B A$
by
$$
\F G_{x,\alpha}(g,h):=\int_\gamma g^*\alpha - \alpha
$$
where $\gamma$ is a path from $x$ to $hx$.
Recall that the expression $\int_\gamma\sigma$
denotes the natural pairing of a chain $\gamma$
and a cochain $\sigma$. 
We find this nonstandard notation useful because in many 
concrete examples presented in the paper the cocycle
$\alpha$ is defined by the integration of a differential
form over smooth paths.

The following basic properties are straightforward to prove.

\begin{lemma}\label{L:G}
\  
\begin{enumerate}
\item
The function $\F G_{x,\alpha}$ is a two-cocycle on $\Homeo(X,\F a)$.
That is it satisfies the following identity:
$$
\F G_{x,\alpha}(h,k)-\F G_{x,\alpha}(gh,k)+
\F G_{x,\alpha}(g,hk)-\F G_{x,\alpha}(g,h)=0.
$$
\item
The value $\F G_{x,\alpha}(g,h)$ does not depend on the
choice of a path from $x$ to $hx$.
\item
The cohomology class of the cocycle $\F G_{x,\alpha}$
depends neither on the choice of the reference
point $x$ nor on the choice of the cocycle~$\alpha$.
\item If either $g$ preserves $\alpha$ or $h$ preserves $x$
then $\F G_{x,\alpha}(g,h)=0$.
\qed
\end{enumerate}
\end{lemma}

\subsection{On singular one-cocycles}\label{SS:singular}
The results of this section are used to prove Lemma \ref{L:continuous}.

Since $H^1(X;\B A)=\Hom(\pi_1(X),\B A)$ one can define a cover
$$
\B A\to X_\F a:=\widetilde X\times_{\pi_1(x)}\B A\to X,
$$
where $\widetilde X$ is the universal cover of $X$ and $\pi_1(x)$
acts on $\B A$ via homomorphism defined by $\F a$.
In what follows, the action $\B A\times X_{\F a}\to X_{\F a}$ by 
the deck transformations will be denoted additively: $(a,z)\mapsto a+z$.

Let $x\in X$ be a reference point in $X$ and let
$\tilde x\in p^{-1}(x)$ be a reference point
in $X_{\F a}$.
Let $\alpha$ be a singular cocycle representing the class $\F a$.
That is, $\alpha$ is a homomorphism $C_1(X;\B A)\to \B A$
defined on the group of chains on $X$ with the coefficients
in $\B A$. It defines an $\B A$-equivariant map
$\B a\colon X_\F a\to\B A$ in the following way.
Given a point $\tilde y\in p^{-1}(y)$ let $\gamma\colon [0,1]\to X$
be a path from $x$ to $y$. Let $\tilde\gamma\colon [0,1]\to X_{\F a}$
be its lift such that $\tilde\gamma(0)=\tilde x$.  Then
we define $\B a\left(\tilde y\right)$ as the unique element such that
$\int_\gamma\alpha+\tilde y
=\B a\left(\tilde y\right)+\tilde\gamma(1)$.
If we put $\tilde y := \tilde\gamma(1)$ we obtain that
$$
\B a(\tilde\gamma(1)) = \int_\gamma\alpha.
$$

Let us check that $\B a$ does not depend on the choice
of the path $\gamma$. Let $\gamma_\pm$ be two paths
from $x$ to $y$ and let $\B a_-$ and $\B a_+$
denote the corresponding maps.
By letting $\tilde y = \tilde\gamma_+(1)$ in the
equality
$$
\int_{\gamma_+}\alpha +\tilde\gamma_-(1)=
\int_{\gamma_-}\alpha +\tilde\gamma_+(1)
$$
we get
$$
\int_{\gamma_+}\alpha +\tilde\gamma_-(1)=
\int_{\gamma_-}\alpha +\tilde y
$$
which shows that
$\B a_+(\tilde y)=\int_{\gamma_+}\alpha = \B a_-(\tilde y)$
as claimed.

The equivariance of $\B a$ is immediate from the definition.
Another choice of a reference point results in changing $\B a$
by an additive constant.

Let $\B a\colon X_\F a\to\B A$ be an $\B A$-equivariant function.
Let $\gamma\colon [0,1]\to X$ be a path and let
$\tilde\gamma\colon [0,1]\to X_a$ be its lift. The following
formula defines a singular one-cocycle with values in $\B A$.
$$
\int_\gamma\alpha=
\B a\left(\tilde\gamma(1)\right)-
\B a\left(\tilde\gamma(0)\right)
$$

\begin{lemma}\label{L:singular}
The above constructions are inverse to each other and
hence provide a bijective correspondence between
singular one-cocycles in the class $\F a\in H^1(X,\B A)$
and $\B A$-equivariant maps $\B a\colon X_\F a\to \B A$
up to the constants.
\end{lemma}

\begin{proof}
Let $\alpha $ be a singular one-cocycle representing the class $\F a$.
It defines an equivariant map $\B a\colon X_{\F a}\to \B A$
such that $\int_\gamma\alpha+\tilde y
=\B a\left(\tilde y\right)+\tilde\gamma(1)$
for every path $\gamma\colon [0,1]\to X$ from $x$ to $y$.
We need to check that
$
\int_\gamma\alpha=
\B a\left(\tilde\gamma(1)\right)-
\B a\left(\tilde\gamma(0)\right)
$.

Let $\tilde y:=\tilde\gamma(1)$ where the lift
$\tilde\gamma$ is chosen so that $\B a(\tilde\gamma(0))=0$.
Then
$$
\int_\gamma\alpha+\tilde\gamma(1)=
\B a\left(\tilde\gamma(1)\right)+\tilde\gamma(1)
$$
implies that $
\int_\gamma\alpha=\B a\left(\tilde\gamma(1)\right)$.

Conversely, let $\B a\colon X_{\F a}\to \B A$ be an $\B A$-equivariant
map. It defines a singular cocycle $\alpha$ by the identity
$\int_\gamma\alpha = \B a(\tilde\gamma(1))$, where
$\tilde\gamma$ is a lift of $\gamma$ such that
$\B a(\tilde\gamma(0))=0$. We then clearly get
that $\int_\gamma\alpha+\tilde\gamma(1)=
\B a\left(\tilde\gamma(1)\right)+\tilde\gamma(1)$.
\end{proof}

\subsection{The one-cocycle $\F K_{\alpha}$}\label{SS:one-cocycle}
If $f\in G\subset\Homeo(X,\F a)$ then $f^*\alpha - \alpha$
is an exact singular one-cocycle on $X$ and the identity
$\delta(\F K_{\alpha}(f))=f^*\alpha - \alpha$
defines a map
$$
\F K_{\alpha}\colon G\to C^0(X;\B A)/\B A.
$$
It is straightforward to check that $\F K_{\alpha}$ is a one-cocycle
(cf. \cite[Proposition 2.3]{GK1}). That is, it satisfies
$$
\F K_{\alpha}(fg)=\F K_{\alpha}(f)\circ g + \F K_{\alpha}(g)
$$
for all $f,g\in G$. 

\begin{lemma}\label{L:G-K}
If $h$ and $g$ are homeomorphisms preserving $\F a=[\alpha]$ then
$$
\F G_{x,\alpha}(g,h)=\F K_\alpha(g)(hx)-\F K_\alpha(g)(x).
$$
\qed
\end{lemma}

\begin{proof}
It is an immediate consequence the definition of the 
cocycle~$\F K_{\alpha}$. Indeed, we have
$$
\F G_{x,\alpha}(g,h)=\int_{x}^{hx}g^*\alpha - \alpha 
=\int_{x}^{hx}\delta(\F K_{\alpha}(g))
=\F K_\alpha(g)(hx)-\F K_\alpha(g)(x).
$$
\end{proof}

\begin{lemma}\label{L:continuous}
Assume that $X$ is paracompact.
Let $\F a\in H^1(X;\B R)$. There exists a singular cocycle $\alpha$ representing
the class $\F a$ such that for any homeomorphism $h\in \Homeo(X,\F a)$
the function $\F K_\alpha(h)$ is a continuous function.
\end{lemma}

\begin{remark}
If $X$ is a differentiable manifold and $\B A=\B R$ then every
cohomology class is represented by a smooth and closed differential
form $\alpha$. It follows that for any diffeomorphism
$h\in\Diff(X,\F a)$ the function $\F K_\alpha(h)$ is smooth.
\end{remark}

\begin{proof}[Proof of Lemma \ref{L:continuous}]
Let us consider the real numbers $\B R$ endowed with with the usual order
topology and consider the bundle
$$
\B R\to E=\widetilde X\times_{\pi_1 X}\B R\stackrel{p}\to X.
$$
Since the fibre is contractible and the base is paracompact
it admits a continuous section $s\colon X\to E$. Such a section
defines a continuous equivariant function
$\B a\colon E\to \B R$ by the identity
$p(\tilde x)=\B a(\tilde x)+sp(\tilde x)$.
Notice that $X_\F a=\widetilde X\times_{\pi_1 X}\B R^\delta$ is the same set
as $E$ but with a finer topology.  Thus $\B a\colon X_\F a\to\B R$ is still
a continuous function.

Let  $\tilde g\in \Homeo(X_{\F a})$ be an $\B R$-equivariant lift
of $g\in\Homeo(X,\F a)$. Define a continuous function
$\widehat{\F K}(g)\colon X_{\F a}\to \B R$ by
$$
\widehat{\F K}(g)(\tilde x) :=
\B a\left(\tilde g\tilde x\right)-\B a(\tilde x).
$$
Since both $\tilde g$ and $\B a$ are $\B R$-equivariant the function
$\widehat{\F K}(g)$ is $\B R$-invariant and thus descends to a
continuous function $\F K(g)\colon X\to \B R$.

Let us show that $\F K=\F K_\alpha$.  Let $\gamma$ be a path between $x$
and $y$.  Let $\tilde\gamma$ be its lift with endpoints at
$\tilde x$ and $\tilde y$.  Then
\begin{align*}
\F K(g)(y)-\F K(g)(x)
&=\left(\B a(\tilde g\tilde y)-\B a(\tilde g\tilde x)\right)-\left(\B a(\tilde y)-\B a(\tilde x)\right)\\
&=\int_{g\gamma}\alpha-\int_\gamma\alpha=\int_\gamma g^*\alpha-\alpha.
\end{align*}
\end{proof}

\section{Distortion in groups}\label{S:distortion}
\subsection{Quasimorphisms}\label{SS:q-m}
Let $\F q\colon G\to \B R$ be a map defined on a group $G$.
The {\bf defect} $D(\F q)$ of the map $\F q$ is defined to be
$$
D(\F q):=\sup_{g,h\in G}|\F q(g) - \F q(gh) + \F q(h)|.
$$
If the defect of $\F q$ is finite then $\F q$ is called
a {\bf quasimorphism}. A quasimorphism $\F q$ is called
{\bf homogeneous} if $\F q(g^n)=n\F q(g)$ for all $n\in \B Z$
and $g\in G$. For every quasimorphism $\F q$ the formula
$$
\widehat {\F q}(g):=\lim_{n\to\infty}\frac{\F q(g^n)}{n}
$$
defines a homogeneous quasimorphism called the homogenisation
of~$\F q$. Moreover, $|\widehat {\F q}(g)-\F q(g)|\leq D$
for all $g\in G$ \cite[Lemma 2.21]{MR2527432}. Thus $\F q$ is unbounded
if and only if so is its homogenisation. 

\begin{proposition}\label{P:q-m}
Let $\F a\in H^1(X;\B R)$.
Let $G\subseteq \Homeo(X,\F a)$ be a subgroup on which
the cocycles $\F G_x$ and $\F G_{y}$ are bounded,
for some $x,y\in X$.
Then the map
$\F q\colon G\to \B R$ defined by
$$
\F q(g):= \F K_{\alpha}(g)(y) - \F K_{\alpha}(g)(x)
$$
is a quasimorphism on $G$ and
$D(\F q)\leq \|\F G_x-\F G_y\|\leq\|\F G_x\|+\|\F G_y\|$,
where $\|\cdot \|$ denotes the supremum norm of a bounded
function. 
\end{proposition}

\begin{proof}
This is a straightforward computation using the cocycle identity for $\F K_{\alpha}$.
\begin{align*}
\F q(f) - \F q(fg) + \F q(g) =&
\F K_{\alpha}(f)(y) - \F K_{\alpha}(f)(x)\\
&-(\F K_{\alpha}(f)(gy) + \F K_{\alpha}(g)(y)  - \F K_{\alpha}(f)(gx) - \F K_{\alpha}(g)(x))\\
&+\F K_{\alpha}(g)(y) - \F K_{\alpha}(g)(x)\\
=&\F K_{\alpha}(f)(gx) - \F K_{\alpha}(f)(x) - (\F K_{\alpha}(f)(gy) - \F K_{\alpha}(f)(y))\\
=&\F G_x(f,h) - \F G_y(f,g).
\end{align*}
\end{proof}

\begin{example}\label{E:torus}
In this example we show that the boundedness of $\F G_x$ depends
on the choice of a point $x\in X$.
Let $X=\B R/\B Z\times \B R\cup\{\infty\}$ be the two-dimensional torus.
Let $\alpha$ be a singular one-cocycle defined by
$$
\int_{\gamma}\alpha := \tilde{\gamma}(1)-\tilde{\gamma}(0),
$$
where $\tilde{\gamma}\colon [0,1]\to \B R$ is a lift of
the composition of $\gamma$ followed by the projection
onto $\B R/\B Z$. Let $\F a$ be the class of $\alpha$.

Let $g\in \Homeo(X,\F a)$ be a homeomorphism defined by
$$
g(t,x):= (t+|x+1|-|x|,x+1).
$$
Then $\F K_{\alpha}(g^n)(t,x) = |x+n| - |x|$ and it follows that
\begin{align*}
\F G_{(0,0)}(g^m,g^n) &=
\F K_{\alpha}(g^m)(g^n(0,0))-\F K_{\alpha}(g^m)(0,0)\\
&= |m+n| - |n| -|m|.
\end{align*}
This shows that $\F G_{(0,0)}$ is unbounded (in fact, the
cocycle $\F G_{(t,x)}$ is un\-boun\-ded whenever $x$ is finite).
On the other hand, $g$ acts trivially on the
circle $\B R/\B Z\times \{\infty\}$ and hence
$\F G_{(t,\infty)}=0$.
\hfill $\diamondsuit$
\end{example}

\begin{remark}\label{R:set}
Let $\psi\colon G\to \Homeo(X,\F a)$ be an action.
The set $\Sigma_{\psi}$ consisting of points $x$ for which the cocycle
$\psi^*\F G_x$  is a bounded is an invariant of the action. If $X$ is compact
then it depends on $\F a$ but not on a continuous representative $\alpha$.
The above example shows that it can be proper. In Section~\ref{SS:top_rot} 
we provide an example of an action for which $\Sigma_{\psi}=~X$.
\end{remark}


Let $X$ be a compact space.
Let us define a pseudo-norm of an element $g\in\Homeo(X,\F a)$
by
$$
\|g\|_\alpha:=
\sup_{x,y\in X}|\F K_{\alpha}(g)(y) - \F K_{\alpha}(g)(x)|.
$$
This means that $\|\cdot\|_{\alpha}$ is symmetric and
satisfies the triangle inequality.
The finiteness of $\|g\|_{\alpha}$ is a consequence of
the compactness of $X$ according to Lemma \ref{L:continuous}.
It follows that if $\Gamma\subset\Homeo(X,\F a)$ is a subgroup
generated by a finite set $S$ then
$$
C\cdot|g|\geq \|g\|_\alpha,
$$
where $C:=\max \{\|s\|_\alpha| s\in S\}$ and $|g|$ denotes the word norm
of $g\in \Gamma$. This is just a special case of the standard
and straightforward to prove fact that 
any pseudo-norm on a group is Lipschitz 
with respect to the word norm.

The next theorem is the main result of this section. Recall that,
according to Lemma \ref{L:G-K} we have that
$$
\F G_{x,\alpha}(g,h)=\F K_{\alpha}(g)(hx)-\F K_{\alpha}(x).
$$
Moreover, it follows from Proposition \ref{P:q-m} that 
if $\F G_x$ and $\F G_y$ are bounded on
the cyclic group $\langle g\rangle$ generated by a homeomorphism $g$ then
the map $\F q\colon \langle g\rangle \to \B R$ defined by
$$
\F q(g^n) = \F K_{\alpha}(g^n)(y)-\F K_{\alpha}(g^n)(x)
$$
is a quasimorphism.

\begin{theorem}\label{T:q-m}
Let $\F a\in H^1(X;\B R)$ and let $g\in \Homeo(X,\F a)$
and assume that $X$ is compact.
Suppose that for some points $x,y \in X$ the cocycles
$\F G_x$ and $\F G_y$ are bounded on the cyclic subgroup
$\langle g\rangle \subset \Homeo(X,\F a)$ ge\-ne\-rated by $g$.
If the above quasimorphism $\F q\colon \langle g\rangle \to \B R$ 
is unbounded then $g$ is undistorted in $\Homeo(X,\F a)$.
\end{theorem}

\begin{remark}
It is often the case that to prove that an element $g$ is 
undistorted in a group $G$ one constructs a
homogeneous quasi-morphism $\F q\colon G\to \B R$ such that
$\F q(g)\neq 0$. Constructing such a quasi-morphism is
in general very difficult. The advantage of the above theorem
is that we only need to check that a naturally defined
quasi-morphism on a cyclic group is unbounded.
\end{remark}

\begin{proof}[Proof of Theorem \ref{T:q-m}]
Let $\Gamma'\subseteq \Homeo(X,\F a)$ be a group containing
$g$ and generated by a finite set $S'\subseteq \Gamma'$.
Consider the subgroup $\Gamma\subseteq \Homeo(X,\F a)$
generated by $S'$ and $h$. It is finitely generated by
the set $S:=S'\cup \{h\}$.

Let $\widehat{\F q}\colon \langle g\rangle\to \B R$ 
be the homogenisation of the quasi-morphism $\F q$.
The following calculation of the translation length of $g$
shows that $g$ is undistorted in $\Gamma$ and hence also
in $\Gamma'\subset \Gamma$. 
\begin{align*}
C\cdot\tau(g)& = \lim_{n\to \infty}\frac{C\cdot|g^n|}{n}\\
& \geq \lim_{n\to \infty}\frac{\|g^n\|_{\alpha}}{n}\\
& \geq \lim_{n\to \infty}\frac{|\F q(g^n)|}{n}\\
& \geq \lim_{n\to \infty}\frac{n|\widehat{\F q}(g)|-D}{n}\\
& = |\widehat{\F q}(g)| > 0.
\end{align*}
Since $\Gamma'$ is an arbitrary finitely generated subgroup of
$\Homeo(X,\F a)$, the element $g$ is undistorted in $\Homeo(X,\F a)$.
\end{proof}

\subsection{Proof of Theorem \ref{T:distortion}}\label{SS:proof_distortion}
Recall that we need to prove that if $x$ and $y$ are fixed points of $g$
and $\int_\gamma g^*\alpha-\alpha\neq 0$ then $g$ is undistorted in
$\Homeo(X,\F a)$. 

First observe that the cocycles $\F G_x$ and $\F G_y$
vanish identically on the cyclic group $\langle g \rangle$ because 
$x$ and $y$ are fixed points of $g$. By Proposition \ref{P:q-m} 
the defect of $\F q$ is zero (since it is bounded by $\|\F G_x\|+\|\F G_y\|=0$) 
and we obtain that $\F q\colon \langle g\rangle \to \B R$ 
is a homomorphism of groups. Furthermore
\begin{align*}
\F q(g)&=\F K_{\alpha}(g)(y) - \F K_{\alpha}(g)(x)\\
&= \int_{\gamma}g^*\alpha -\alpha \neq 0
\end{align*}
according to the hypothesis. Therefore $\F q$ is unbounded and
the statement follows from Theorem \ref{T:q-m}.
\qed

\subsection{Bounded cocycles}\label{SS:bounded}
In what follows we are interested in bounded cohomology
of a group with the integer coefficients; see Gromov
\cite{MR686042} and Monod \cite{MR1840942} for a background
on bounded cohomology.

\begin{example}
(Ghys \cite[Section 6.3]{MR1876932})
\label{E:boundedZ}
The second bounded cohomology $H^2_{\OP{b}}(B\B Z;\B Z)$
of the integers with integer coefficients is isomorphic
to $\B R/\B Z$. To see this let
$\F c \colon \B Z \times \B Z\to \B Z$
be a bounded two-cocycle. As an ordinary cocycle it is
a coboundary since $H^2(B\B Z;\B Z)=0$.
If $\F c=\delta\F b$ then, since $\F c$ is
bounded, the cochain $\F b$ is a quasimorphism.
The homogenisation of $\F b$ (which is a real cochain in general)
is given by $\widehat {\F b}(n)=rn$ for some real number $r\in \B R$.
The required isomorphism
$$
H^2_{\OP{b}}(B\B Z;\B Z)\to \B R/\B Z
$$
is defined by $[\F c]\mapsto r +\B Z$.
\hfill $\diamondsuit$
\end{example}

Let $g\in \Homeo(X,\F a)$ and let $x\in X$ be a point for which the
cocycle $\F G_x$ is bounded on the cyclic group generated by $g$.
The cohomology class $r_x(g)+\B Z\in H^2(\langle g\rangle;\B Z)$ 
represented by the pullback of $\F G_x$ is called the 
{\bf local rotation number} of $g$ at the point $x\in X$.

Let us explain the geometry of the local rotation number.
Take a path $\eta_{x,1}\colon [0,1]\to~X$ from $x$ to $gx$
and let $\eta_{x,n}$ be the concatenation of paths
$g^k(\eta_{x,1})$ for $k$ ranging from $0$ to $n-1$.
Define a map $\F b_x\colon \langle g\rangle \to \B R$
by
$$
\F b_x(g^n):=-\int_{\eta_{x,n}}\alpha.
$$
Observe that $\delta\F b_x = \F G_x$ on the cyclic group
$\langle g\rangle$. Since $\F G_x$ is bounded on $\langle g\rangle$
we get that $\F b_x$ is a quasi-morphism and that its
homogenisation satisfies $\widehat {\F b}_x(g^n)=r_x(g)n$ for
a suitable representative of the local rotation number of $g$ at $x$.
This shows that there exists a constant $C_x>0$ such that
$$
|\F b_x(g^n) - r_x(g)n|\leq C_x  
$$
for all $n\in \B Z$. We thus obtain that
$$
\lim_{n\to \infty} \frac{\F b_x(g^n)}{n} = r_x(g)
$$
and hence the above limit represents the local rotation
number of $g$ at $x$. Notice that, since $\alpha$ has integral
periods, the dependence of $\F b_x(g^n)$ on the choice of
the path $\eta_{x,1}$ is up to an integer constant only.
This implies that the above computation of the
local rotation number does not depend on the choice
of a path $\eta_{x,1}$.

\begin{remark}
If $X=\B S^1$ then the cocycle $\F G$ corresponding to
the length form is equal to the Euler
cocycle. Consequently, the local rotation number defined above
equals the classical topological rotation number of a
homeomorphism of the circle \cite[Section 6.3]{MR1876932}.
\end{remark}

\subsection{Proof of Theorem \ref{T:rotation}}\label{SS:proof_rotation}
In order to apply Theorem \ref{T:q-m} we need to prove that the quasimorphism 
$\F q\colon \langle g\rangle \to \B R$ from Proposition \ref{P:q-m} is unbounded.

Let $\gamma,\eta_{x,n},\eta_{y,n}\colon [0,1]\to X$
be  paths from $x$ to $y$, $x$ to $g^nx$ and $y$
to $g^ny$ respectively and $n\in \B Z$. As above assume that
$\eta_{x,n}$ is a concatenation of the paths $g^k (\eta_{x,1})$
for $k$ ranging from $0$ to $n-1$ and similarly for $\eta_{y,n}$.

Let $\F b_x,\F b_y\colon G\to \B R$ be defined as above and
let $\square_n$ be a concatenation of $-\gamma$, $\eta_{x,n}$,
$g^n \gamma$ and $-\eta_{y,n}$. We get the following computation.
\begin{align*}
\F q(g^n) &= \int_{\gamma} (g^n)^*\alpha - \alpha\\
&= \int_{\square_n}\alpha - \int_{\eta_{x,n}}\alpha + \int_{\eta_{y,n}}\alpha\\
&= n \int_{\square_1}\alpha + \F b_x(g^n) - \F b_y(g^n)\\
&= n \left( \int_{\square_1}\alpha + (r_x(g)-r_y(g))\right ) + O(1).
\end{align*}
Since $\alpha$ has integral periods and the difference $r_x(g)-r_y(g)\notin \B Z$
by the hypothesis, we get that the quasi-morphism $\F q$ is unbounded. 
\qed

\subsection{Some consequences of Theorem \ref{T:rotation}}
\label{SS:top_rot}
\begin{corollary}\label{C:distorted}
If $g\in \Homeo(X;\F a)$ is distorted then the local rotation
number is constant at all points $x$ for which $\F G_x$ is bounded.
\qed
\end{corollary}

\begin{example}\label{R:gradient}
If $X$ is a closed oriented surface of a positive genus then 
$g\in \Homeo(X)_0$ has a fixed point and hence if it is distorted then
the local rotation number of $g$ has to vanish.
\hfill
$\diamondsuit$
\end{example}

\begin{example}
If $g$ is a time-one map of a gradient flow then $\F G_x$
is bounded at every $x$ and its local rotation number is equal to zero.
We do not know whether such elements are distorted or not.
\hfill
$\diamondsuit$
\end{example}

Let $G\subset \Homeo(X)$ be a group of homeomorphisms
acting trivially on $H^1(X;\B R)$.
Let $g$ be a homeomorphism distorted in $G$.
Let $\ell_1,\ell_2\subset X$ be oriented simple closed curves
preserved by $g$. We also assume that the classes $[\ell_i]$ are
nonzero in $H^1(X;\B R)$. Let $\rho_1$ and $\rho_2$ be the
topological rotation numbers associated with the action
of $g$ on $\ell_1$ and $\ell_2$ respectively.

\begin{proposition}\label{P:top_rot}
With the above notation and assumptions
the following statements are true.
\begin{enumerate}
\item
There exist nonzero integers $k_1,k_2\in \B Z$ such that
$k_1\rho_1 = k_2\rho_2$ in $\B Q/\B Z$.
\item
If the classes $[\ell_1]$ and $[\ell_2]$
are linearly independent in $H^1(X;\B R)$ then
both $\rho_1$ and $\rho_2$ are rational.
\end{enumerate}
\end{proposition}

\begin{proof}
Let $\alpha$ be an integral singular one-cocycle and let $x_i\in \ell_i$.
We get that the topological and the local rotation numbers are
related as follows
$$
\rho_i\cdot \int_{\ell_i}\alpha = r_{x_i}(g).
$$
Since $g$ is distorted, it follows from Theorem \ref{T:rotation}
that the local rotation numbers $r_{x_i}(g)$ are equal.
Choosing $\alpha$ such that
$\int_{\ell_i}\alpha \neq 0$ proves the first statement.

To prove the second assertion, we choose $\alpha$ such that
$\int_{\ell_1}\alpha = 0 \neq \int_{\ell_2}\alpha$. It follows
that $\rho_2\cdot \int_{\ell_2}\alpha=0$ and, since
$\int_{\ell_2}\alpha$ is an integer, it implies that
$\rho_2 \in \B Q/\B Z$. The rationality of $\rho_1$ is
proven similarly.
\end{proof}

\section{Vanishing properties of the cocycle $\F G$}\label{S:vanish}
In this section we prove theorems about the vanishing of
the cohomology class represented by the cocycle $\F G$.

\subsection{Proof of Theorem \ref{T:subset}}\label{SS:proof_subset}
(Cf. proof of Theorem 1.3 in \cite{GK2}.)
Recall that $i_L\colon L\to M$ is the inclusion of a path-connected
subset such that $i_L^*\F a =0$.
Therefore, for any $\alpha$ in the cohomology class $\F a$ there exists
a map $F\colon L\to\B A$ such that
$\int_\gamma\alpha=F(\gamma(1))-F(\gamma(0))$
provided $\gamma$ is a path contained in $L$.

Let us choose a reference point $x$ in $L$.
Since the homeomorphism  $h\in G$ preserves $L$ and $L$ 
is path-connected there exists a path in $L$ from $x$ to $hx$.
Therefore
$$
\F G(g,h)=F(ghx)-F(gx)-\left(F(hx)-F(x)\right).
$$
If we define $\F b(g):=F(x)-F(gx)$ then we get that
$$
\F G(g,h)=\F b(g)-\F b(gh)-\F b(h)=\delta\F b(g,h),
$$
and hence the cocycle $\F G$ is cohomologically trivial.
\qed

\begin{example}
The class $\F G$, which is equal to 
the Euler class of $\Homeo(\B S^1)_0$, is nontrivial on 
$\PSl(2,\B R)$ and restricts to a nontrivial class on any 
cocompact lattice
$\Gamma\subset \PSl(2,\B R)$ \cite[Section 6.2]{MR1876932}.

On the other hand,
every orbit of $\Gamma$ is countable thus simply-connected.
This shows that the assumption on the connectivity
of an invariant subset is essential in Theorem \ref{T:subset}.
\hfill $\diamondsuit$
\end{example}

\subsection{Proof of Theorem \ref{T:measure}}\label{SS:proof_measure}
Recall that we need to prove that if an action of a group
$G$ preserves a Borel probability measure then the corresponding
class $\psi^*[\F G]$ is trivial.

Let us choose, by Lemma \ref{L:continuous}, a representative $\alpha$
of $\F a\in H^1(X,\B R)$ such that $\F K_\alpha(h)$ is continuous
for each homeomorphism $h$ of $X$.
Let $\mu$ be a Borel probability measure.
We define the lift $\widetilde{\F K}_\alpha(h)$ by the following
normalisation condition,
$$
\int_X\widetilde{\F K}_\alpha(h)\mu=0.
$$
This can be done because $\F K_\alpha(h)$, being continuous,
is integrable.

Since $G$ preserves a measure $\mu$ we see that
$$
\widetilde{\F K}_\alpha(g)\circ h-\widetilde{\F K}_\alpha(gh)
+\widetilde{\F K}_\alpha(h)=0.
$$
Indeed, the left hand side is a constant and integrating it
with respect to $\mu$ we get that this constant
is zero. Thus $\widetilde {\F K}_{\alpha}$ is a one-cocycle
with values in $C^0(X;\B R)$ lifting the cocycle $\F K_{\alpha}$.

Finally, by Lemma \ref{L:G-K} we have that
$$
\F G(g,h)
=-\widetilde{\F K}_\alpha(g)(x)+\widetilde{\F K}_\alpha(g)(hx)
=-\widetilde{\F K}_\alpha(g)(x)+\widetilde{\F K}_\alpha(gh)(x)
-\widetilde{\F K}_\alpha(h)(x).
$$
This implies that  $\F G=\delta\F b$ for a cochain 
defined by
$\F b(g):=-\widetilde{\F K}_\alpha(g)(x)$.
\qed

\section{Cohomology class defined by $\F G$}\label{S:class}
Recall that we have defined the cocycle $\F G$ by an explicit
formula in Section \ref{SS:explicit} and also the cocycle
$\F K_{\alpha}$ in Section \ref{SS:one-cocycle}. The purpose of this
section is to give several characterisations of the cohomology
class represented by $\F G$. This will be used in the
next section to prove the nonvanishing results for
the cohomology class $[\F G]$.

Throughout this section let $G:=\Homeo(X,\F a)$ denote the group 
of homeomorphisms of $X$ preserving the cohomology 
class $\F a\in H^1(X;\B A)$.

\subsection{The extension class}
Let $\B A\to X_{\F a}\stackrel{p}\to X$ be the covering associated
with the cohomology class ${\F a}\in H^1(X;\B A)$. 
Let $G_\F a$ be the group of homeomorphisms of $X_\F a$
commuting with the deck transformations and projecting
onto $G$.
There is a central extension
$$
0\to\B A\to G_\F a\to G\to 0.
$$
Let $\F E_\F a \in H^2(BG^{\delta},\B A)$ be the corresponding
extension class.

\subsection{Transgression}
Consider the following universal fibration
$$
X\to X_{G^{\delta}}:=
X\times_{G^{\delta}}EG^{\delta}\to BG^{\delta}
$$
associated with the natural action of $G^{\delta}$ on $X$.
Then the differential
$$
d_2\colon E_2^{0,1}=H^1(X,\B A)^G\to H^2(BG^\delta,\B A)=E_2^{2,0}
$$
in the Leray-Serre spectral sequence defines
a cohomology class $d_2[\F a]\in H^2(BG^{\delta};\B A)$.

\begin{theorem}\label{T:def}
The following equalities hold in\/ $H^2(BG^\delta,\B A)$
$$
\delta[\F K_\alpha]=[\F G]=\F E_\F a=d_2\F a.
$$
\end{theorem}

\subsection{Proof of the first equality: $\delta[\F K_{\alpha}]=[\F G]$}
Consider the following extension of $G$-representations
$$
0\to\B A \to C^0(X;\B A)\to C^0(X;\B A)/\B A\to 0.
$$
It induces the connecting homomorphism
$$
\delta \colon H^1\left(BG^\delta;C^0(X;\B A)/\B A\right)
\to H^2\left(BG^\delta;\B A\right)
$$
and hence we obtain the class
$\delta[\F K_{\alpha}]\in H^2(BG^\delta;\B A)$.

Take a section $C^0(X,\B A)/\B A\to C^0(X,\B A)$ that chooses
a function vanishing at the basepoint $x$. Denote by
$\widetilde{\F K}_{\alpha}$ the lift of the cocycle to $C^0(X,\B A)$
using this section. The function
$$
(g,h)\mapsto\widetilde{\F K}_{\alpha}(g)\circ h -
\widetilde{\F K}_{\alpha}(gh) + \widetilde{\F K}_{\alpha}(h)
$$
is constant and equal to $(\partial\widetilde{\F K}_{\alpha})(g,h)$, the
value of a cocycle representing the class $\partial [\F K_{\alpha}]$.
Evaluating it at the basepoint we get that
$$
\partial\widetilde{\F K}_{\alpha}(g,h) = \widetilde{\F K}_{\alpha}(g)(hx)
=\int_\gamma g^*\alpha - \alpha = \F G_x(g,h),
$$
where $\gamma$ is any path between $x$ and $hx$.
\qed

\subsection{Proof of the second equality: $[\F G]=\F E_{\F a}$}
Let $\B a\colon X_{\F a}\to \B A$ be an equivariant map
representing the cohomology class $\F a$ as in Lemma \ref{L:singular}.
Fix a reference point $\tilde x\in p^{-1}(x)\subset X_\F a$ with
$\B a\left(\tilde x\right)=0$.
Consider the extension
$\B A\to G_{\F a}\to G$ and let
$\widetilde{\cdot}\colon G \to G_{\F a}$
be a section defined by
\begin{equation}\label{Eq:section}
\B a\left(\tilde f\tilde x\right)=0.
\end{equation}
Since $\tilde f$ commutes with the action of $\B A$, it follows
that if $\tilde y_1$ and $\tilde y_2$ are two points in the
same fibre of $X_\F a\to X$ we have
\begin{equation}\label{Eq:fibre}
\B a\left(\tilde f\tilde y_1\right)-
\B a\left(\tilde y_1\right)=
\B a\left(\tilde f\tilde y_2\right)-
\B a\left(\tilde y_2\right).
\end{equation}
Let $\gamma\colon [0,1]\to X$ be a curve from $x$ to $gx$.
Let $\tilde\gamma\colon [0,1]\to X_\F a$ be a lift with
$\tilde\gamma(0)=\tilde x$.
Let $\tilde y_1=\tilde g\tilde x$
and $\tilde y_2=\tilde\gamma(1)$.
Then we obtain the following equalities
\begin{equation}\label{Eq:34}
\B a\left(\tilde f\tilde g\tilde x\right)=
\B a\left(\tilde f\tilde g\tilde x\right)-
\B a\left(\tilde g\tilde x\right)=
\B a\left(\tilde f\tilde\gamma(1)\right)-
\B a\left(\tilde\gamma(1)\right).
\end{equation}
where the first equality follows from (\ref{Eq:section}) and the
second one from (\ref{Eq:fibre}).

Since $\tilde f\tilde\gamma$ is a lift of $f\gamma$ starting at
$\tilde f\tilde\gamma(0)=\tilde f\tilde x$, we have that
$$
\B a\left(\tilde f\tilde\gamma(1)\right)=\int_{f\gamma}\alpha,
$$
where $\alpha$ and $\B a$ are related as in
Lemma \ref{L:singular}.
This together with (\ref{Eq:34}) implies that
\begin{equation}\label{Eq:35}
\B a\left(\tilde f\tilde g\tilde x\right)=
\int_{f\gamma}\alpha-\int_\gamma\alpha=\F G(f,g).
\end{equation}

The extension class is represented by the cocycle defined
by the following identity \cite[Section IV.3]{MR83k:20002}:
\begin{equation}\label{Eq:brown}
\F E(f,g) + \widetilde{fg}=\tilde f \tilde g.
\end{equation}

In the following calculation the second equality
follows from the equivariance of $\B a$ and the
others as marked.
\begin{align*}
&&\F E(f,g) &=\F E(f,g)+\B a\left(\widetilde{fg}\tilde x\right)&&\rlap{by (\ref{Eq:section})}\\
&&          &=\B a\left(\F E(f,g)+\widetilde{fg}\tilde x\right)&&\\
&&          &=\B a\left(\tilde f\tilde g\tilde x\right)&&\rlap{by (\ref{Eq:brown})}\\
&&          &=\F G(f,g)&&\rlap{by (\ref{Eq:35})}
\end{align*}

\subsection{Proof of the third equality: $\F E_{\F a}=d_2[\F a]$}
The Lyndon-Hochschild-Serre spectral sequence associated to
the extension $\B A\to G_{\F a}\to G$ is isomorphic to the
Leray-Serre spectral sequence associated with the
fibration
$$
B\B A \to B{G}^\delta_{\F a}\to BG^\delta.
$$
We shall make all the computations in the latter.
We have isomorphisms
$$
H^0(BG^\delta;H^1(B\B A;\B A))
=H^1(B\B A;\B A)^G
=\Hom(\B A;\B A)^G
$$
and, using this identifications, the extension class is defined to be
$$
\F E_{\F a}:=d_2[\OP{id}]
$$
where
$d_2\colon H^0(BG^\delta;H^1(B\B A;\B A))\to
H^2(BG^\delta;H^0(B\B A;\B A))=
H^2(BG^\delta;\B A)$ is the differential
in the spectral sequence.


\begin{proposition}\label{P:ext}
Let $X\to E:=EG^\delta\times_{G^\delta} X
\to BG^\delta$ be the universal bundle associated
with the action of\/ $G^\delta$ on $X$.
Then the extension class
$$
\F E_{\F a}=d_2({\F a})
$$
where $d_2\colon H^1(X;\B A)^{G}\to H^2(BG^\delta;\B A)$ is
the differential in the associated spectral sequence.
\end{proposition}

\begin{proof}
Since there is an isomorphism $H^1(X;\B A)=[X,B\B A]$,
the class $\F a$ can be represented
by a continuous map $\alpha\colon X\to B\B A=K(\B A,1)$.
Thus $\F a=\alpha^*[\OP{id}]$.
Let
$$
E_{\F a}:=EG_{\F a}^\delta\times_{G_{\F a}^\delta} X_{\F a}
$$
and let us consider the following diagram of fibrations.
$$
\xymatrix
{
\B A\ar[d]\ar[r] &X_{\F a}\ar[d]\ar[r] &X \ar[d]\ar[r]^\alpha &B \B A\ar[d]\\
E\B A\ar[d]\ar[r] &E_{\F a}\ar[d]\ar[r] &E \ar[d]\ar[r]^{\F A} &BG^\delta_{\F a}\ar[d]\\
B\B A\ar[r] &BG_{\F a}^\delta\ar[r] &BG^\delta \ar@{=}[r] &BG^\delta\\
}
$$
Since the fibration $E\B A\to E_{\F a}\to E$ has a contractible
fibre, it admits a section. The map $\F A\colon E\to BG^\delta_{\F a}$
is defined as the composition of this section followed by the
projection. In this way the two right-hand side columns form
a morphism of bundles. Hence the result follow from the
functoriality of the spectral sequence.
\end{proof}

\section{Non-triviality of the cocycle $\F G$}\label{S:nonvanish}
\subsection{Proof of Theorem \ref{T:ev}}\label{SS:proof_ev}
The idea of the proof is that the cohomology class
$\psi^*[\F G]$ is equal to the transgression of $\F a$ in
the universal bundle $X\to X_{G^{\delta}}\to BG^{\delta}$
according to Theorem \ref{T:def}. This implies
that it is the image of the corresponding transgression
for the universal bundle $X\to X_G\to BG$ for the
action of the connected group $G$. The transgression
for a connected group is then related with the
topology of the corresponding evaluation map.

More precisely, the identity homomorphism 
$G^\delta\to G$ induces a continuous map
$\beta\colon BG^\delta\to BG$. Consider the
following morphism of universal bundles.
$$
\xymatrix
{
X\ar@{=}[r]\ar[d] & X\ar[d]\\
X_{G^{\delta}}\ar[r]\ar[d] & X_G\ar[d]\\
BG^\delta\ar[r]^{\beta}& BG\\
}
$$
It follows from Theorem \ref{T:def} 
that $\psi^*[\F G]=\beta^*(d_2({\F a}))$.

Let $\OP{ev}\colon G\to X$ denote the evaluation at the reference
point $x\in X$ associated with the action. We have yet another
morphism of universal bundles.
$$
\xymatrix
{
G\ar[r]^{\OP{ev}}\ar[d] & X\ar[d]\\
EG\ar[r]\ar[d] & X_G\ar[d]\\
BG\ar@{=}[r]& BG\\
}
$$
The evaluation map induces the morphism of
spectral sequences that on the second page
is the map
$$
\C E^{p,q}\colon H^p(BG;H^q(X;\B A))\to H^p(BG;H^q(G;\B A))
$$
defined by $\OP{ev}^*$ on the coefficients level.
It follows from the connectivity of $G$ that $\C E^{p,0}$ is
injective. In both spectral sequences we denote the differential
on the second page by~$d_2$. We have the
following straightforward equalities.
$$
d_2(\OP{ev}^*({\F a})) = d_2(\C E^{0,1}({\F a})) =
\C E^{2,0}(d_2({\F a}))
$$
Since $EG$ is contractible the corresponding differential
$$
d_2\colon H^0(BG;H^1(G;\B A))\to H^2(BG;H^0(G;\B A))
$$
is an isomorphism. This, together with the injectivity
of $\C E^{p,0}$, implies that
$\OP{ev}^*({\F a})\neq 0$ if and only if
$d_2({\F a})\neq 0$ which finishes the proof of
Theorem~\ref{T:ev}.
\qed

\subsection{Remark on Corollary \ref{C:homeo}}
The result of McDuff in \cite{MR569248} states
that the comparison map
$\beta\colon B\Homeo(X)^\delta\to B\Homeo(X)$
is a homology equivalence for the full group of
homeo\-morphisms. It is however known that that
the homotopy fibre of $\beta$ is determined
by the topological and algebraic structure of
a neighbourhood of the identity. This implies
that $\beta$ is a homology equivalence for
any group of homeomorphisms containing the
connected component of the identity as a subgroup.
In particular it holds for $\Homeo(X,{\F a})$.

\subsection{Proof of Corollary \ref{C:perfect}}\label{SS:perfect}
Let $\widetilde G\to G$ be the universal cover.
It follows from the countability of the fundamental
group $\pi_1(G)$ that $\widetilde{G}$ is perfect.
Indeed, let
$\OP{Ab}\colon \widetilde{G} \to H:=
\widetilde G/\left[\widetilde G,\widetilde G\right]$
be the abelianisation. It induces a surjective map
$G\to H/\OP{Ab}(\pi_1(G))$ which is trivial
because $G$ is perfect. This implies that
$H=\OP{Ab}(\pi_1(G))$ and since $H$ is path
connected it must be trivial.

Now we follow the proof of Lemma 6 in Milnor \cite{MR699007}.
Let $\C F$ denote the homotopy fibre of the comparison
map $\beta\colon BG^\delta\to BG$. Since it depends
only on the local structure of the group it is also
the homotopy fibre of the corresponding comparison
map for the universal cover. It follows from
the perfectness of $\widetilde G$ and the spectral
sequence for the fibration
$\C F\to B\widetilde G^{\delta}\to B\widetilde G$
that $H^1(\C F;\B Z)=0$.
This  implies that $H^1(\C F;\B A)=0$ by the universal
coefficients theorem.
It then follows from the spectral
sequence for the fibration
$\C F \to BG^\delta\to BG$ that the homomorphism
$\beta^*\colon H^2(BG;\B A)\to H^2(BG^\delta;\B A)$
is injective.
\qed

\subsection*{Acknowledgements}
We thank Alessandra Iozzi, Assaf Libman and Aleksy Tralle for helpful 
comments and discussions.

\'S.R. Gal is partially supported by Polish MNiSW grant N N201 541738
and Swiss NSF Sinergia Grant CRSI22-130435.

\bibliography{../../bib/bibliography}
\bibliographystyle{acm}

\end{document}